\documentclass[11pt, a4paper, reqno]{amsart}
\usepackage[a4paper, total={6in, 9in}]{geometry}
\usepackage{amsfonts}   
\usepackage{amsmath}
\usepackage[utf8]{inputenc}
\usepackage[english]{babel}
\usepackage{amssymb}
\usepackage{tikz-cd}
\usepackage{comment}
\usepackage{amsthm}
\usepackage{xcolor}
\usepackage[normalem]{ulem}
\newtheorem{theorem}{Theorem}[section]

\newtheorem{lemma}{Lemma}[section]

\newtheorem*{claim*}{Claim}
\newtheorem*{proofclaim*}{Proof of Claim}

\theoremstyle{theorem}

\usepackage{embedfile}

\newtheorem*{Exercise 11.3.B}{Exercise 11.3.B}

\numberwithin{equation}{section}

\usepackage{fancyhdr}

\usepackage{graphicx}
\graphicspath{ {./images/} }

\usepackage{hyperref}                                    
\usepackage{enumerate}

\usepackage[all,cmtip]{xy}

\usepackage{blindtext}

\begin{document}
	\title[Identities of the product of Dirichlet Series ]
	{Identities for the product of Two Dirichlet Series Satisfying Hecke's Functional Equation}
	\author{Bruce C.~Berndt, Likun Xie}
	\address{Department of Mathematics, University of Illinois, 1409 West Green
		Street, Urbana, IL 61801, USA} \email{berndt@illinois.edu}
	\address{Department of Mathematics, University of Illinois, 1409 West Green
		Street, Urbana, IL 61801, USA}
	\email{likunx2@illinois.edu}

	\begin{abstract}
		We derive a general formula for the product of two Dirichlet series that satisfy Hecke's functional equation. Several  examples are provided to demonstrate the applicability of the formula.  In addition, we discuss prior work on similar products and clarify certain issues arising in the existing literature.
	\end{abstract}
	\maketitle	
	
	\section{Introduction}
	
	In their study of the zeros of the Riemann zeta function \(\zeta(s)\), Hardy and Littlewood \cite{HL1, HL2, HL3} developed approximate functional equations for \(\zeta(s)\) and \(\zeta^2(s)\). Throughout this paper, let \(s = \sigma + it\) and (where relevant) \(s' = \sigma' + it'\), with \(\sigma, \sigma', t\), and \(t'\) being real. As customary, define
	\[
	\chi(s) := 2(2\pi)^{s-1}\sin\left(\tfrac{1}{2} \pi s\right)\Gamma(1-s),
	\]
	where \(\Gamma(s)\) denotes the gamma function.
	
	\begin{theorem}\label{thm1} \textup{\textbf{(Approximate Functional Equation for $\zeta(s)$)}}
		Let \(x\), \(y\), \(h\), and \(k\) be positive real numbers such that
		\[
		2\pi xy = |t|, \quad x > h > 0, \quad y > k > 0, \quad -k < \sigma < k.
		\]
		Then, as \(|t| \to \infty\),
		\[
		\zeta(s) = \sum_{n \leq x} n^{-s} + \chi(s)\sum_{n \leq y} n^{s-1} + O(x^{-s}) + O(|t|^{1/2-\sigma}y^{\sigma-1}).
		\]
	\end{theorem}
	
	Theorem \ref{thm1} can be regarded as a "compromise" between the series representation for \(\zeta(s)\),
	\[
	\zeta(s) = \sum_{n=1}^{\infty} n^{-s}, \quad \sigma > 1,
	\]
	and the functional equation
	\[
	\zeta(s) = \chi(s)\zeta(1-s).
	\]
	
	Hardy and Littlewood's approximate functional equation for \(\zeta^2(s)\) \cite[Theorem 2]{HL3} is as follows:
	
	\begin{theorem}\textup{\textbf{(Approximate Functional Equation for $\zeta^2(s)$)}}
		If
		\[
		-\frac{1}{2} \leq \sigma \leq \frac{3}{2}, \quad x > A, \quad y > A, \quad xy = \left(\frac{t}{2\pi}\right)^2,
		\]
		then
		\[
		\zeta^2(s) = \sum_{n \leq x} \frac{d(n)}{n^s} + \chi(s) \sum_{n \leq y} \frac{d(n)}{n^{1-s}} + O\left(x^{1/2-\sigma}\left(\frac{x+y}{t}\right)^{1/2} \log t\right),
		\]
		where \(d(n)\) is the number of divisors of \(n\).
	\end{theorem}
	
	Since the aforementioned works of Hardy and Littlewood, approximate functional equations have been derived for other Dirichlet series satisfying similar functional equations. For instance, Apostol and Sklar \cite{apostol-sklar} established an analogue of Theorem \ref{thm1} for Hecke’s Dirichlet series. References for additional approximate functional equations can be found in \cite{apostol-sklar}.
	
	In 1930, Wilton \cite{wilton-2} derived an approximate functional equation for the product of two Riemann zeta functions \(\zeta(s)\) and \(\zeta(s')\), which might be considered a second "compromise."
	
	\begin{theorem} \cite{wilton-2}
		Let
		\[
		\sigma_k(n) := \sum_{d \mid n} d^k,
		\]
		and assume
		\[
		-\frac{1}{2} \leq \sigma \leq \frac{3}{2}, \quad -\frac{1}{2} \leq \sigma' \leq \frac{3}{2}, \quad \sigma' > \sigma - 1,
		\]
		\[
		\tau := \frac{|t|}{2\pi} \geq \tau' := \frac{|t'|}{2\pi} > A > 1,
		\]
		where \(A\) is a fixed constant. If
		\[
		\tau' > A\sqrt{\tau}, \quad \textup{when} \quad \sigma' \leq \sigma,
		\]
		and
		\[
		\tau' > A\tau^{(1-(\sigma'-\sigma))/2}, \quad \textup{when} \quad \sigma < \sigma' < 1+\sigma,
		\]
		then, as \(|t| \to \infty\),
		\[
		\zeta(s)\zeta(s') = \sum_{n \leq \tau} \frac{\sigma_{s-s'}(n)}{n^s} + \chi(s)\chi(s') \sum_{n \leq \tau'} \frac{\sigma_{s'-s}(n)}{n^{1-s}} + \mu\log\tau\left\{O(\tau^{1/2-\sigma}) + O(\tau^{1/2-\sigma'})\right\},
		\]
		where
		\[
		\mu := \min\left(|\sigma-\sigma'|^{-1}, \log\tau\right).
		\]
	\end{theorem}
	
	A key component in Wilton's derivation of this functional equation is the identity \cite[Equation (2.2)]{wilton-2}
	\begin{align}\label{wilton}
		&\zeta(s)\zeta(s') - \zeta(s+s'-1) \left(\frac{1}{s-1} + \frac{1}{s'-1}\right)\nonumber \\=&
		2(2\pi)^{s-1} \sum_{n=1}^\infty \sigma_{1-s-s'}(n)n^{s-1}s\int_{2\pi n}^\infty u^{-s-1}\sin(u)du \nonumber
		\\&+ 2(2\pi)^{s'-1} \sum_{n=1}^\infty \sigma_{1-s-s'}(n)n^{s'-1}s'\int_{2\pi n}^\infty u^{-s'-1}\sin(u)du.
	\end{align}
	
	Alternative proofs of \eqref{wilton} were attempted in \cite{nakajima} and \cite{linearized}, but both contain gaps. 
	In order to obtain analogous formulas for the product of two Hurwitz zeta functions \cite{hurwitz} and the product of two Dedekind zeta functions \cite{wilton}, similar mistakes were made.
	Furthermore, another paper \cite{banerjee_Piltz} employing a related approach involving the Riesz sum  also contains a gap concerning the validity of termwise differentiability.
	We further discuss these papers at the end of Section \ref{section_3}. We are not aware of any additional papers establishing approximate functional equations for products of Dirichlet series.
	
	In this paper, we derive a formula for the product of two Dirichlet series satisfying Hecke-type functional equations.
	
	\section{Background Set-up}
	\subsection{Dirichlet Series with a General Functional Equation}\label{functional}
	
	We consider Dirichlet series associated with a general functional equation of Hecke’s type, as defined in \cite{functional}. Let \(\{\lambda_n\}\) and \(\{\mu_n\}\) be two sequences of positive real numbers satisfying:
	\[
	0 < \lambda_1 < \lambda_2 < \cdots < \lambda_n \to \infty, \quad 0 < \mu_1 < \mu_2 < \cdots < \mu_n \to \infty,
	\]
	with the additional property that $$\lambda_k  \lambda_l=\lambda_m \iff kl=m, \quad \text{ and }\quad\mu_k \mu_l=\mu_m\iff kl=m.$$
	
	Let \(\{f(n)\}\) and \(\{g(n)\}\) be two nonzero sequences of complex numbers. Denote the complex variable \(s\) as \(s = \sigma + it\), where \(\sigma\) and \(t\) are real numbers, and let \(\delta\) be a positive real constant.

	Suppose  that \(\varphi(s)\) and \(\psi(s)\) are Dirichlet series defined by
	\[
	\varphi(s) := \sum_{n=1}^\infty \frac{f(n)}{\lambda_n^s}, \quad \sigma > \sigma_{a, \varphi}, \quad \text{and} \quad \psi(s) := \sum_{n=1}^\infty \frac{g(n)}{\mu_n^s}, \quad \sigma > \sigma_{a, \psi},
	\]
	where \(\sigma_{a, \varphi}\) and \(\sigma_{a, \psi}\) are the finite abscissae of absolute convergence of \(\varphi(s)\) and \(\psi(s)\), respectively. Assume \(\varphi(s)\) and \(\psi(s)\) extend meromorphically across the entire complex plane and satisfy the functional equations
	\begin{equation} \label{functional_1}
		(2\pi)^{-s} \Gamma(s) \varphi(s) = (2\pi)^{s-\delta} \Gamma(\delta-s) \psi(\delta-s),
	\end{equation}
	\begin{equation}\label{functional_2}
		(2\pi)^{-s} \Gamma(s) \psi(s) = (2\pi)^{s-\delta} \Gamma(\delta-s) \varphi(\delta-s),
	\end{equation}
	where \(\delta\) is a positive real number.
	
	These functional equations hold in the sense that there exist a domain \(D\), which is the exterior of a bounded closed set \(S\) in the \(s\)-plane, and a holomorphic function \(\chi(s)\) on \(D\) with the property
	\[
	\chi(\sigma + it) e^{-\epsilon|t|} = O(1), \quad 0 < \epsilon < \frac{\pi}{2},
	\]
	as \(|t| \to \infty\), uniformly in any strip \(\sigma_1 \leq \sigma \leq \sigma_2\), where \(-\infty < \sigma_1 < \sigma_2 < +\infty\). Additionally, assume that
	\[
	\chi(s) =
	\begin{cases}
		(2\pi)^{-s} \Gamma(s) \varphi(s), & \text{for } \sigma > \alpha, \\
		(2\pi)^{s-\delta} \Gamma(\delta-s) \varphi(\delta-s), & \text{for } \sigma < \beta,
	\end{cases}
	\]
	where \(\alpha\) and \(\beta\) are constants.
	
	We further assume that \(\varphi(s)\) and \(\psi(s)\) have finitely many poles, all of which are simple and located on the real axis. This assumption is satisfied by most arithmetic functions of interest.
	If \(\varphi(s)\) has a pole, we assume \(\sigma_{a, \varphi} \leq \delta\); otherwise, we allow \(\delta < \sigma_{a, \varphi}\). Similar conditions apply to \(\psi(s)\).
	
	Let \(p_{\varphi,1}, p_{\varphi,2}, \dots, p_{\varphi,k_\varphi}\) denote the poles of \(\varphi(s)\), with corresponding residues \(r_{\varphi,1}, r_{\varphi,2}, \dots, r_{\varphi,k_\varphi}\). Similarly, let \(p_{\psi,1}, p_{\psi,2}, \dots, p_{\psi,k_\psi}\) and \(r_{\psi,1}, r_{\psi,2}, \dots, r_{\psi,k_\psi}\) denote the poles and residues of \(\psi(s)\).
	\subsection{Meijer $G$-Function}
	
	The Meijer $G$-function is defined as a complex integral \cite[p.~207]{bateman}
	\begin{equation}\label{meijer}
		G^{m,n}_{p,q} \left( \, {}^{a_1,\dots,a_p}_{b_1,\dots,b_q} \, \middle| \, z \right) = \frac{1}{2\pi i} \int_L \frac{\prod_{j=1}^m \Gamma(b_j - s) \prod_{j=1}^n \Gamma(1 - a_j + s)}{\prod_{j=m+1}^q \Gamma(1 -b_j + s) \prod_{j=n+1}^p \Gamma(a_j -s)} z^s \, ds,
	\end{equation}
	where an empty product is interpreted as \(1\), and the parameters satisfy \(0 \leq m \leq q\), \(0 \leq n \leq p\). The parameters are chosen such that all poles of the integrand are simple, i.e., no pole of \(\Gamma(b_j - s)\), \(j = 1, \dots, m\), coincides with any pole of \(\Gamma(1 - a_k + s)\), \(k = 1, \dots, n\).
	
	The path of integration \(L\) is taken to be a vertical line running from \(-i\infty\) to \(+i\infty\), such that all poles of \(\Gamma(b_j - s)\), \(j = 1, \dots, m\), lie to the right of \(L\), and all poles of \(\Gamma(1 - a_k + s)\), \(k = 1, \dots, n\), lie to the left of \(L\) \cite[(2), p.~207]{bateman}.
	The integral converges for \(|\arg z| < \sigma \pi\), where
	\[
	\sigma = m + n - \frac{1}{2}(p + q), \quad \sigma > 0.
	\]
	Additionally, the integral converges for \(|\arg z| = \sigma \pi\) if the following condition is satisfied
	\begin{equation}\label{converging_condition}
		(q - p)\left(\Re s + \frac{1}{2}\right) > \Re \nu + 1,
	\end{equation}
	where
	\[
	\nu = \sum_{j=1}^q b_j - \sum_{j=1}^p a_j.
	\]
	
	Throughout the paper, all instances of the Meijer $G$-function are of the form  
	\begin{equation}\label{GG}
		G^{1,1}_{3,1} \left( \, {}^{1,b,c}_{d} \, \middle| \, z \right)  = \frac{1}{2\pi i} \int_L \frac{\Gamma(d - s) \Gamma(s)}{\Gamma(b - s) \Gamma(c - s)} z^s \, ds.
	\end{equation}
	An anonymous referee kindly observed that \eqref{GG} can be expressed in terms of the hypergeometric function
	\[
	{}_1F_2\left(d; b, c; z \right) := \sum_{n=0}^{\infty} \frac{(d)_n}{n!(b)_n(c)_n}z^n, \quad |z| < \infty,
	\]
	where the Pochhammer symbol is defined as
	\[
	(a)_0 := 1, \quad (a)_n := a(a+1) \cdots (a+n-1), \quad n \geq 1.
	\]
	Using \cite[p.~1099, no.~8]{gr}, we have the identity
	\[
	G^{1,1}_{3,1} \left( \, {}^{1, b, c}_{d} \; \middle| \; z \right)  
	= \frac{1}{2\pi i} \int_{-L} \frac{\Gamma(d + t) \Gamma(-t)}{\Gamma(b + t) \Gamma(c + t)} 
	\left(\frac{1}{z}\right)^{-t} dt 
	= \frac{\Gamma(d)}{\Gamma(b) \Gamma(c)} \, {}_1F_2\left(d; b, c; -\frac{1}{z} \right).
	\]
	For compactness of notation, we shall use the $G$-function representation instead of hypergeometric functions in the remainder of the paper.

	\subsection{Integral with Riesz Sum}
	
	We first present the following form of Perron's formula.
	
	\begin{lemma}[{\cite[Lemma 1]{functional}; see also \cite[p.~81]{means}}]
		Assume that the Dirichlet series
		\[
		\varphi(s) = \sum_{n=1}^\infty \frac{a_n}{\lambda_n^s}, \quad \sigma > \sigma_{a, \varphi},
		\]
		has a finite abscissa of absolute convergence, denoted by \(\sigma_{a, \varphi}\). Then, for \(k \geq 0\), \(\alpha > 0\), and \(\alpha \geq \sigma_{a, \varphi}\), we have
		\begin{equation}
			\frac{1}{\Gamma(k+1)}\sideset{}{'}\sum_{\lambda_n \leq x} a_n (x - \lambda_n)^k
			= \frac{1}{2\pi i} \int_{\alpha - i\infty}^{\alpha + i\infty}
			\frac{\Gamma(s)}{\Gamma(s + k + 1)} \varphi(s) x^{s + k} \, ds,
		\end{equation}
		where the prime on the summation indicates that the last term is weighted by \(\frac{1}{2}\) if \(k = 0\) and \(x = \lambda_n\).
	\end{lemma}

	Define the integral
	\begin{equation}\label{defi_integral}
		S_\alpha(k, x) := \frac{1}{2\pi i} \int_{\alpha-i\infty}^{\alpha+i\infty}
		\frac{\Gamma(s)}{\Gamma(s+k+1)} \varphi(u+s)\psi(v-s)x^{s+k} \, ds.
	\end{equation}
	Applying the lemma above to the function \(F(s) = \varphi(u+s)\psi(v-s)\), we obtain the following result.
	
	\begin{lemma}\label{lemma_sum}
		Assume  that the Dirichlet series
		\[
		\varphi(s) = \sum_{m=1}^\infty \frac{f(m)}{\lambda_m^s}, \quad \sigma > \sigma_{a, \varphi},
		\quad \text{and} \quad
		\psi(s) = \sum_{n=1}^\infty \frac{g(n)}{\mu_n^s}, \quad \sigma > \sigma_{a, \psi},
		\]
		have finite abscissae of absolute convergence, denoted by \(\sigma_{a, \varphi}\) and \(\sigma_{a, \psi}\), respectively. Then, for \(k \geq 0\), \(\alpha + \Re u > \sigma_{a, \varphi}\), \(\Re v - \alpha > \sigma_{a, \psi}\), and \(\alpha > 0\),
		\[
		S_\alpha(k, x) = \frac{1}{\Gamma(k+1)} \sum_{n=1}^\infty \frac{g(n)}{\mu_n^{v+k}}
		\sideset{}{'}	\sum_{\lambda_m \leq \mu_n x} \frac{f(m)}{\lambda_m^u} (\mu_n x - \lambda_m)^k.
		\]
	\end{lemma}

	\section{A Formula for the Product \(\varphi(u)\psi(v)\)}\label{section_3}
	
	\begin{theorem}\label{main_theorem}
		Let \(\varphi\) and \(\psi\) be Dirichlet series satisfying the setup in \ref{functional}. Let \(C\) be a curve encircling all elements of \(S\), i.e., all poles of \(\varphi\) and \(\psi\). Suppose \(\gamma\) is a real number such that
		\[
		\gamma > \max(\sigma_{a, \varphi}, \sigma_{a, \psi}), \quad \gamma > \frac{\delta}{2},
		\]
		and the curve \(C\) is entirely contained within the strip \(\delta - \gamma < \Re s < \gamma\). Let \(k\) be an integer such that \(k > 2\gamma - \delta\). Suppose that \(u\) and \(v\) are complex numbers such that \(\Re u > \sigma_{a, \varphi}\), \(\Re v > \sigma_{a, \psi}\), and
		\[
		\Re u + \Re v - \gamma > \max(\sigma_{a, \varphi}, \sigma_{a, \psi}).
		\]
		Then,
		\begin{align}
			\varphi(u)\psi(v) &= -\sum_{i=1}^{n_\psi} \frac{r_{\psi, i}}{p_{\psi, i} - v} \varphi(v + u - p_{\psi, i})
			- \sum_{i=1}^{n_\varphi} \frac{r_{\varphi, i}}{p_{\varphi, i} - u} \psi(v + u - p_{\varphi, i}) \nonumber \\
			&\quad - \left. \frac{d^{(k)}}{dx} \right|_{x=1} G_{f, k}(v, u, x)
			- \left. \frac{d^{(k)}}{dx} \right|_{x=1} G_{g, k}(u, v, x), \label{id_1}
		\end{align}
		where
		\[
		G_{f, k}(v, u, x) := \frac{(2\pi)^\delta}{2\pi i} \sum_{n=1}^\infty f_{\delta - v - u}(n) x^{\delta - v + k}
		G^{1, 1}_{3, 1}\left(
		{}^{1, \delta - v + k + 1, \delta}_{\delta - v} \, \middle| \, \frac{1}{4\pi^2 \lambda_n x}
		\right),
		\]
		\[
		f_a(n) := \sum_{d \mid n} \lambda_d^a f(d) f\left(\frac{n}{d}\right),
		\]
		\[
		G_{g, k}(u, v, x) := \frac{(2\pi)^\delta}{2\pi i} \sum_{n=1}^\infty g_{\delta - u - v}(n) x^{\delta - u + k}
		G^{1, 1}_{3, 1}\left(
		{}^{1, \delta - u + k + 1, \delta}_{\delta - u} \, \middle| \, \frac{1}{4\pi^2 \mu_n x}
		\right),
		\]
		\[
		g_a(n) := \sum_{d \mid n} \mu_d^a g(d) g\left(\frac{n}{d}\right),
		\]
		and \(p_{\varphi, i}, p_{\psi, i}\) denote the poles of \(\varphi\) and \(\psi\), with \(r_{\varphi, i}, r_{\psi, i}\) being the corresponding residues.
		
		Moreover, for any \(x > 0\),
		\begin{align}
			\frac{1}{\Gamma(k+1)} \sum_{n=1}^\infty \frac{g(n)}{\mu_n^{v + k}}
			\sideset{}{'}\sum_{m \leq nx} \frac{f(m)}{\lambda_m^u} (nx - m)^k
			&= \frac{1}{2\pi i} \int_C \frac{\Gamma(z - u)}{\Gamma(z - u + k + 1)} \varphi(z) \psi(v + u - z) x^{z - u + k} \, dz \nonumber \\
			&\quad + G_{g, k}(u, v, x), \label{id_2}
		\end{align}
		\begin{align}
			\frac{1}{\Gamma(k+1)} \sum_{m=1}^\infty \frac{f(m)}{\lambda_m^{u + k}}
			\sideset{}{'}\sum_{n \leq mx} \frac{g(n)}{\mu_n^v} (mx - n)^k
			&= \frac{1}{2\pi i} \int_C \frac{\Gamma(z - v)}{\Gamma(z - v + k + 1)} \psi(z) \varphi(v + u - z) x^{z - v + k} \, dz \nonumber \\
			&\quad + G_{f, k}(v, u, x). \label{id_3}
		\end{align}
	\end{theorem}

	\begin{proof}

		Recall that $S_{\alpha,k}(k,x)$ is defined by \eqref{defi_integral}. By a change of variable $z= u+s$, we have
		\begin{align*}
			S_\alpha
			(k, x) &= \frac{1}{2\pi i } \int_{\alpha-i\infty}^{\alpha+i\infty} \frac{\Gamma(s)}{\Gamma(s+k+1) }\varphi(u+s)\psi(v-s)x^{s+k} ds\\
			&= \frac{1}{2\pi i } \int_{\alpha+\Re u-i\infty}^{\alpha+\Re u+i\infty} \frac{\Gamma(z-u)}{\Gamma(z-u+k+1) }\varphi(z)\psi(v+u-z)x^{z-u+k} dz.
		\end{align*}
		Applying Lemma \ref{lemma_sum} above, we have
		\begin{align}
			S_{\gamma-\Re u }
			(k, x)  &=\frac{1}{2\pi i } \int_{\gamma-i\infty}^{\gamma+i\infty} \frac{\Gamma(z-u)}{\Gamma(z-u+k+1) }\varphi(z)\psi(v+u-z)x^{z-u+k} dz\notag\\
			&=\frac{1}{\Gamma(k+1 )}\sum_{n=1}^\infty \frac{g(n)}{\mu_n ^{v+k}} \sideset{}{'}\sum_{m\leq nx} \frac{f(m)}{\lambda_m^u} (\mu_n x- \lambda_m) ^k.\label{s1}
		\end{align}
		We shall move the line of integration to $z=\delta-\gamma +it,  -\infty <t<+\infty$, by applying Cauchy's residue theorem  to a rectangle with vertices $\gamma \pm iT$, and $\delta-\gamma\pm iT$, and then letting $T\to\infty$. We shall show that the contribution from integration on the two horizontal sides $\sigma+iT,\delta-\gamma \leq \sigma \leq \gamma$, tends to zero as $T\to\infty$.  We will apply the Phragmen-Lindelof principle \cite[Theorem 4.1]{conway} to the half-strip $\delta-\gamma \leq \sigma \leq \gamma$, $t\geq  T$. First we check the boundary condition.
		
		On the line $z=\gamma+it$, by the choice of \(\gamma\), we have \(\varphi(z)\psi(v + u - z) = O(1)\). Recall Stirling's formula
		\begin{equation}\label{stirling}
			\Gamma(s)\sim\sqrt{2\pi}s^{s-1/2}e^{-s}, \quad |t| \to \infty.
		\end{equation}
		Hence,
		\[
		\frac{\Gamma(z - u)}{\Gamma(z - u + k + 1)} = O(|t|^{-k-1}) ,\quad |t|\to\infty.
		\]

		On the line 	$z=\delta -  \gamma+it$, applying the functional equation  \eqref{functional_1}, we deduce that
		\begin{align*}
			&	\frac{\Gamma(\delta-\gamma-u)}{\Gamma(\delta-\gamma-u+k+1) }\varphi(\delta-\gamma)\psi(v+u-\delta+\gamma)x^{\delta-\gamma-u+k}\\ &= 	\frac{\Gamma(\delta-\gamma-u)}{\Gamma(\delta-\gamma-u+k+1) }
			(2\pi)^{\delta-2\gamma}\frac{\Gamma(\gamma)}{\Gamma(\delta-\gamma)}
			\psi(\gamma)\psi(v+u-\delta+\gamma)x^{\delta-\gamma-u+k}\\
			&= O(|t|^{-k-1+2\gamma-\delta}).
		\end{align*}
		
		Since $k> 2\gamma -\delta$, on both lines $z=\gamma +it $ and $z=\delta -\gamma +it$, as $|t|\to\infty$, the integrand satisfies
		\[
		\frac{\Gamma(z-u)}{\Gamma(z-u+k+1) }\varphi(z)\psi(v+u-z)x^{z-u+k}= O(|t|^{-1}).
		\]	
		By hypothesis, the integrand is analytic in the half-strip \(\delta - \gamma \leq \sigma \leq \gamma\), \(t > T\), and satisfies the growth condition
		\[
		\frac{\Gamma(z - u)}{\Gamma(z - u + k + 1)} \varphi(z) \psi(v + u - z) x^{z - u + k} \ll e^{a |t|}, \quad \text{for some } a > 0,
		\]
		uniformly in the strip \(\delta - \gamma \leq \sigma \leq \gamma\).

		By the Phragmén–Lindelöf principle \cite[Theorem 4.1]{conway}, it follows that \[
		\frac{\Gamma(z-u)}{\Gamma(z-u+k+1) }\varphi(z)\psi(v+u-z)x^{z-u+k}=O(|t|^{-1}),
		\]
		as $|t|\to \infty,$   in the strip $\delta-\gamma \leq \sigma \leq \gamma.$ Hence, the contributions  from the integration along the two horizontal sides vanish as $T\to\infty.$

		Denote
		\begin{equation}
			g_a(n) := \sum_{d \mid n} {\mu_{d}}^a g(d) g\left(\frac{n}{d}\right),\label{g}
		\end{equation}
		Thus, when \(\Re u, \Re v > \sigma_{a,\psi}\), we have
		\[
		\psi(u)\psi(v) = \sum_{n=1}^\infty \frac{g_{u-v}(n)}{\mu_n^u}.
		\]
		
		Define
		\[
		P_k(x) := \frac{1}{2\pi i} \int_C \frac{\Gamma(z-u)}{\Gamma(z-u+k+1)} \varphi(z)\psi(v+u-z)x^{z-u+k} \, dz,
		\]
		where $C$ is the closed curve defined in the hypotheses of our theorem.
		Hence, by \eqref{s1} and \eqref{g},
		\begin{align}
			&\frac{1}{\Gamma(k+1)} \sum_{n=1}^\infty \frac{g(n)}{\mu_n^{v+k}} \sideset{}{'}\sum_{\lambda_m \leq \mu_ nx}  \frac{f(m)}{\lambda_m^u} (\mu_n x - \lambda_m)^k \nonumber \\
			&= P_k(x) + \frac{1}{2\pi i} \int_{\delta-\gamma-i\infty}^{\delta-\gamma+i\infty} \frac{\Gamma(z-u)}{\Gamma(z-u+k+1)} \varphi(z) \psi(v+u-z)x^{z-u+k} \, dz \nonumber \\
			&= P_k(x) + \frac{1}{2\pi i} \int_{\gamma-i\infty}^{\gamma+i\infty}
			\frac{(2\pi)^{\delta-2w} \Gamma(\delta-w-u) \Gamma(w)}{\Gamma(\delta-w-u+k+1) \Gamma(\delta-w)} \psi(w) \psi(v+u-\delta+w)x^{\delta-w-u+k} \, dw \nonumber \\
			&= P_k(x) + \frac{(2\pi)^\delta}{2\pi i} \sum_{n=1}^\infty g_{\delta-u-v}(n) x^{\delta-u+k}
			\int_{\gamma-i\infty}^{\gamma+i\infty}
			\frac{\Gamma(\delta-w-u) \Gamma(w)}{\Gamma(\delta-w-u+k+1) \Gamma(\delta-w)}
			\frac{1}{(4\pi^2 \mu_n x)^w} \, dw \nonumber \\
			&= P_k(x) + \frac{(2\pi)^\delta}{2\pi i} \sum_{n=1}^\infty g_{\delta-u-v}(n) x^{\delta-u+k} G^{1,1}_{3,1}
			\left(
			\begin{array}{c}
				1, \delta-u+k+1, \delta \\
				\delta-u
			\end{array}
			\,\middle|\,
			\frac{1}{4\pi^2 \mu_n x}
			\right), \label{equality_1}
		\end{align}
		where \( G^{1,1}_{3,1} \) denotes the Meijer \( G \)-function, defined in \eqref{meijer}.
		
		The first equality above 
		comes from the contour shift as justified above. The second equality follows from a change of variable $z=\delta-w$ and an application of the functional equation \eqref{functional_1}. In the third equality, the interchange of summation and integration is justified by the absolute convergence of the Dirichlet series for \(\psi(w)\psi(v+u-\delta+w)\) along the line \(w=\gamma\), and the integral itself is absolutely convergent when \(k > 2\gamma - \delta\). The convergence of the integral can be verified either by directly applying Stirling’s asymptotic formula \eqref{stirling} or by employing the conditions in \eqref{converging_condition}, which also stem from Stirling’s asymptotic formula.

		Denote the function
		\begin{equation}\label{s3}
			G_{g,k}(u,v,x) := \frac{(2\pi)^\delta}{2\pi i } \sum_{n=1}^\infty g_{\delta-u-v}(n)x^{\delta-u+k} G^{1,1} _{3,1}\left(  ^{1,\delta -u+k+1, \delta}_{\delta-u }\mid \frac{1}{4\pi^2 \mu_n x }\right),
		\end{equation}
		where $g_a(n)$ is defined in \eqref{g}.

		We now show that  \eqref{id_2} follows from \eqref{equality_1}. Taking the $k$-th derivative of $P_k(x)$, letting $x=1$, and applying the residue theorem, we find that
		\begin{align}
			\frac{d^{(k)}}{dx}\Big|_{{x=1}}P_k(x) &=\frac{1}{2\pi i }\int_C \frac{1}{z-u} \varphi(z) \psi(v+u-z) dz\notag \\
			&= \varphi(u)\psi(v )  + \sum_{i=1} ^{n_\varphi } \frac{r_{\varphi,i}}{p_{\varphi,i}-u}\psi(v+u- p_{\varphi,i}).\label{s2}
		\end{align}
		Therefore, taking the derivative of \eqref{equality_1}, employing \eqref{s2} and \eqref{s3}, and letting $x=1$, we deduce that
		\begin{equation}\label{expresion_1}
			\sum_{n=1}^\infty \frac{g(n)}{\mu_n ^{v}}\sideset{}{'} \sum_{\lambda_m\leq \mu_n} \frac{f(m)}{\lambda_m^u} =\varphi(u)\psi(v )  + \sum_{i=1} ^{n_\varphi } \frac{r_{\varphi,i}}{p_{\varphi,i}-u}\psi(v+u- p_{\varphi,i}) + \frac{d^{(k)}}{dx}\Big|_{{x=1}}G_{g,k}(u,v,x).
		\end{equation}
		
		Similarly, we have
		\begin{equation}\label{expression_2}
			\sum_{m=1}^\infty \frac{f(m)}{\lambda_ m ^{u}} \sideset{}{'}\sum_{\mu_n\leq \lambda_m} \frac{g(n)}{\mu_n^v} =\varphi(u)\psi(v )  + \sum_{i=1} ^{n_\psi } \frac{r_{\psi,i}}{p_{\psi,i}-v}\varphi(v+u- p_{\psi,i}) + \frac{d^{(k)}}{dx}\Big|_{{x=1}}G_{f,k}(v,u,x).
		\end{equation}
		Since \[	\sum_{n=1}^\infty \frac{g(n)}{\mu_n ^{v}} \sideset{}{'}\sum_{\lambda_m\leq \mu_n} \frac{f(m)}{\lambda_m^u}+	\sum_{m=1}^\infty \frac{f(m)}{\lambda_m ^{u}} \sideset{}{'}\sum_{\mu_n\leq \lambda_m} \frac{g(n)}{\mu_n^v} = 	\sum_{m=1}^\infty \frac{f(m)}{\lambda_m ^{u}} \sum_{n=1}^\infty \frac{g(n)}{\mu_n ^{v}},
		\]
		summing  \eqref{expresion_1} and \eqref{expression_2}, we deduce that
		\begin{align*}
			\varphi(u)\psi(v ) = &-\sum_{i=1} ^{n_\psi } \frac{r_{\psi,i}}{p_{\psi,i}-v}\varphi(v+u- p_{\psi,i}) - \sum_{i=1} ^{n_\varphi }\frac{r_{\varphi,i}}{p_{\varphi,i}-u}\psi(v+u- p_{\varphi,i})\\ &-\frac{d^{(k)}}{dx}\Big|_{{x=1}}G_{f,k}(v,u,x)- \frac{d^{(k)}}{dx}\Big|_{{x=1}}G_{g,k}(u,v,x),
		\end{align*}
		which completes the proof of Theorem \ref{main_theorem}.	
	\end{proof}
	
	One might be tempted to take the term-wise derivative of \( G_{f,k}(v,u,x) \) and \( G_{g,k}(u,v,x) \) to derive an expression for \(\varphi(u)\psi(v)\) involving an infinite series. However, this approach is not justifiable, as the integral represented by the Meijer \( G \)-function in each term of the summation does not converge after taking the \(k\)-th derivative.
	
	Specifically, consider the summation term in \eqref{equality_1},
	\begin{align*}
		&g_{\delta-u-v}(n)x^{\delta-u+k} G^{1,1}_{3,1}\left(
		^{1,\delta -u+k+1, \delta}_{\delta-u} \,\middle|\, \frac{1}{4\pi^2 \mu_n x}
		\right)\nonumber \\
		=& g_{\delta-u-v}(n)x^{\delta-u+k} \int_{\gamma-i\infty}^{\gamma+i\infty}
		\frac{\Gamma(\delta -w-u)\Gamma(w)}{\Gamma(\delta-w-u+k+1) \Gamma(\delta-w)} \frac{1}{(4\pi^2\mu_nx)^w} \, dw.
	\end{align*}
	The \(k\)-th derivative with respect to $x$ of this term, with the use of  the functional equation for $\Gamma(z)$, reduces to the term obtained by setting \(k=0\) in the expression above.
	As demonstrated in the proof, the Meijer \( G \)-function in each summation term of \eqref{equality_1} converges if and only if \( k > 2\gamma - \delta \). However, this condition is not satisfied for \(k = 0\), since \(2\gamma - \delta > 0\). This error occurred when the author applied term-wise differentiation in \cite[p.~1034]{banerjee_Piltz}.
	
	In \cite{nakajima}, \cite{linearized}, \cite{hurwitz} and \cite{wilton}, the authors attempted to derive a product formula for two Dirichlet series. However, these proofs contain errors due to unjustified contour shifts, overlooked conditions for integral convergence, and neglect of the validity of infinite series expansions. Similarly, in \cite{banerjee_Piltz}, the author attempted to derive a formula for the sum of 
	$k$-th divisor functions over a number field using the Riesz sum. However, this work contains an error concerning termwise differentiation, as we highlighted earlier.
	
	Starting with \cite{nakajima}, the author attempted to prove Wilton's formula \eqref{wilton} using a different method. However, in their proof, the contour shift is unjustifiable, as they shift the integral from a line where it does not converge to one where it does \cite[p.~21, upper right]{nakajima}. Additionally, their decomposition of the double summation [see Remark, \cite{nakajima}] is incorrect.
	
	In the follow-up paper \cite{linearized}, the decomposition error was corrected, but the issue with the invalid contour shift persisted \cite[p.~125, upper left]{linearized}. In \cite{wilton}, similar difficulties arose in the attempt to establish a formula for the product of two Dedekind zeta functions.
	
	Although a contour shift was not explicitly mentioned in \cite{wilton}, the problem is hidden in the fact that the condition for the integrals \cite[(3.10), (3.11)]{wilton} to converge, as stated at the bottom of \cite[p.~12]{wilton}, is incompatible with the requirements for expanding \( f(-z) \) into a product of Dirichlet series \cite[p.~8]{wilton}. Furthermore, when using this expansion, the authors did not specify the valid range of \( u \) and \( v \), introducing additional ambiguity.

	In \cite{hurwitz}, the same method was applied to derive a product formula for \(\kappa\) Hurwitz zeta functions. Equation \cite[(2.5)]{hurwitz}, found on \cite[p.~34, right]{hurwitz}, represents an integral along the line \(-a\), with \(a\) assumed to satisfy \(-a < 1\). The authors expand the product of \(\kappa\) Hurwitz zeta functions within the integral as a Dirichlet series and exchange the order of integration and summation, leading to two integrals, \(I_1\) (\cite[p.~34, bottom right]{hurwitz}) and \(I_2\) (\cite[p.~35, upper left]{hurwitz}). However, on the line \(\sigma = -a < 1\), the Dirichlet series expansion and interchange of summation and integration are invalid. Additionally, for the \(G_{0,\kappa}^{\kappa,0}\)-function inside the integral to converge, \(a\) must satisfy \(a < -1/2\), a condition not specified in the paper.

	In our work, we derive a product formula for two Dirichlet series using a different contour-shifting approach, which we have  justified. However, due to the inability to perform term-wise differentiation, we cannot obtain a formula similar to Wilton's \eqref{wilton}, where the terms in the infinite series have a specific expression. Despite this limitation, we hope that our formula will prove useful in deriving an approximate functional equation for the product of two Dirichlet series.

	\section{Arithmetical Identities}
	
	We now apply Theorem \ref{main_theorem} to several examples of arithmetical functions. It is worth noting that the range of \(u\) and \(v\) for which these identities are valid can be extended through the analytic continuation of the associated functions.

	\subsection{Dirichlet Series for the Ramanujan $\tau$-function}
	
	Let \(\tau(n)\) denote the Ramanujan tau-function. Then \eqref{functional_1} and \eqref{functional_2} are satisfied with \(\delta = 12\), \(\lambda_n = \mu_n = n\), and \(f(n) = g(n) = \tau(n)\).
	
	Define 
	\[
	\varphi(s) = \psi (s)= \sum_{n=1}^\infty \frac{\tau(n)}{n^s}.
	\]
	The  Dirichlet series converges absolutely for  \(\sigma> 13/2\), and has an analytic continuation to an entire function. Under these conditions, we may take \(\gamma = 13/2 + \epsilon\) for some small \(\epsilon > 0\). For \(k > 2\) and \(\Re u + \Re v > 13\), we have
	\begin{align}
		\varphi(u) \varphi(v) = -\frac{d^{(k)}}{dx}\Big|_{x=1} G_{\tau,k}(u,v,x)
		- \frac{d^{(k)}}{dx}\Big|_{x=1} G_{\tau,k}(v,u,x), \label{id_1tau}
	\end{align}
	where
	\[
	G_{\tau,k}(u,v,x) := \frac{(2\pi)^{12}}{2\pi i} \sum_{n=1}^\infty \tau_{12-u-v}(n) x^{12-u+k} G^{1,1}_{3,1}
	\left(
	\begin{array}{c}
		1, 12-u+k+1, 12 \\
		12-u
	\end{array}
	\,\middle|\,
	\frac{1}{4\pi^2 n x}
	\right),
	\]
	and
	\[
	\tau_a(n) = \sum_{d \mid n} d^a \tau(d) \tau\left(\frac{n}{d}\right).
	\]
	
	\subsection{The Riemann Zeta Function}
	
	Let \(\zeta(s)\) denote the Riemann zeta function. Then \eqref{functional_1} and \eqref{functional_2} are satisfied with \(\delta = \frac{1}{2}\), \(\lambda_n = \mu_n = \frac{n^2}{2}\), \(f(n) = g(n) = 1\), and \(\varphi(s) = \psi(s) = 2^s \zeta(2s)\). The abscissa of absolute convergence is \(\sigma_a = \frac{1}{2}\).
	
	Under these conditions, for \(k > \frac{1}{2}\) and \(\Re u + \Re v > 1\), we have
	\begin{align}
		\varphi(u) \psi(v) &= -\frac{\sqrt{2}}{1 - 2v} \varphi(u+v-\frac{1}{2})
		- \frac{\sqrt{2}}{1 - 2u} \psi(u+v-\frac{1}{2}) \nonumber \\
		&\quad - \frac{d^{(k)}}{dx}\Big|_{x=1} G_f(u,v,x)
		- \frac{d^{(k)}}{dx}\Big|_{x=1} G_{f,k}(v,u,x),
	\end{align}
	where
	\[
	G_{f,k}(u,v,x) := \frac{(2\pi)^{1/2}}{2\pi i} \sum_{n=1}^\infty f_{1/2-u-v}(n) x^{1/2-u+k} G^{1,1}_{3,1}
	\left(
	\begin{array}{c}
		1, \frac{1}{2}-u+k+1, \frac{1}{2} \\
		\frac{1}{2}-u
	\end{array}
	\,\middle|\,
	\frac{1}{2\pi^2 n^2 x}
	\right),
	\]
	and
	\[
	f_a(n) = \sum_{d \mid n} \left(\frac{d^2}{2}\right)^a.
	\]
	
	\subsection{The \(l\)-th Divisor Function}\label{divisor}
	
	Let \(\sigma_l(n)\) denote the sum of the \(l\)-th powers of the divisors of \(n\), where \(l\) is an integer. It is well known that
	\[
	\sum_{n=1}^\infty \frac{\sigma_l(n)}{n^s} = \zeta(s)\zeta(s-l),
	\]
	with abscissa of absolute convergence \(\sigma_a = l + 1\). Let 
	$\varphi(s )=\psi(s) =\zeta(s)\zeta(s-l). $
	Then, \eqref{functional_1} and \eqref{functional_2} are satisfied with \(\delta = l + 1\), \(\lambda_n = \mu_n = n\), \(f(n) = \sigma_l(n)\), and \(g(n) = (-1)^{\frac{l+1}{2}} \sigma_l(n)\), provided \(l\) is odd. Under these conditions, for \(k > l + 1\) and \(\Re u + \Re v > 2l + 2\), we have
	\begin{align}
		\varphi(u) \psi(v) &= -\frac{\zeta(1-l)}{1 - v} \varphi(u+v-1)
		- \frac{\zeta(1+l)}{l+1 - v} \varphi(u+v-l-1) \nonumber \\
		&\quad - (-1)^{\frac{l+1}{2}} \frac{\zeta(1-l)}{1 - u} \psi(u+v-1)
		- (-1)^{\frac{l+1}{2}} \frac{\zeta(1+l)}{l+1 - u} \psi(u+v-l-1) \nonumber \\
		&\quad - \frac{d^{(k)}}{dx}\Big|_{x=1} G_{\sigma_l,k}(u,v,x)
		- \frac{d^{(k)}}{dx}\Big|_{x=1} G_{\sigma_l,k}(v,u,x),
	\end{align}
	where
	\[
	G_{\sigma_l,k}(u,v,x) := \frac{(2\pi)^{l+1}}{2\pi i} \sum_{n=1}^\infty f_{l+1-u-v}(n) x^{l+1-u+k} G^{1,1}_{3,1}
	\left(
	\begin{array}{c}
		1, l+1-u+k+1, l+1 \\
		l+1-u
	\end{array}
	\,\middle|\,
	\frac{1}{4\pi^2 n x}
	\right),
	\]
	and
	\[
	f_a(n) = \sum_{d \mid n} d^a \sigma_l(d) \sigma_l\left(\frac{n}{d}\right).
	\]
	
	This formula provides an identity for \(\zeta(u) \zeta(u-l) \zeta(v) \zeta(v-l)\).
	
	\subsection{Additional Examples}
	
	For further examples of arithmetic functions where our formulas apply, see \cite[Section 5]{functional}.

	\bibliographystyle{plain}

\begin{thebibliography}{99}	
		\bibitem{apostol-sklar}
		Apostol, T.~M.~and Sklar, A.~(1957). The approximate functional equation of Hecke's Dirichlet series, \emph{Trans.~Amer.~Math.~Soc.}~\textbf{86}, 446–462.
		
		\bibitem{banerjee_Piltz} Banerjee, S. (2021). Piltz divisor problem over number fields à la Voronoï. \textit{Proceedings of the American Mathematical Society}, 149(3), 1025-1038.
		
		\bibitem{wilton}
		Banerjee, S., Chakraborty, K., and Hoque, A. (2021). An analogue of Wilton's formula and values of Dedekind zeta functions. \textit{Journal of Mathematical Analysis and Applications}, 495(1), 124675.
		
		\bibitem{linearized}
		Banerjee, D. and Mehta, J. (2014). Linearized product of two Riemann zeta functions. \textit{Proceedings of the Japan Academy}, 90(8), 123.
		
		
		\bibitem{bochner}
		Bochner, S. (1951). Some properties of modular relations. \textit{Annals of Mathematics}, 53(2), 332–363.
		
		
		\bibitem{means}
		Chandrasekharan, K. and Minakshisundaram, S. (1952). Typical means. Oxford University Press.
		
		\bibitem{functional}
		Chandrasekharan, K. and Narasimhan, R. (1961). Hecke's functional equation and arithmetical identities. \textit{Annals of Mathematics}, 74(1), 1–23.
		
		
		\bibitem{conway}
		Conway, J. B. (1973). \textit{Functions of one complex variable}. 2nd ed., Vol.~11, Graduate Texts in Mathematics, Springer.
		
		
		\bibitem{bateman}
		Erd\'{e}lyi, A.~(1955). Higher Transcendental Functions, Vol.~3, McGraw-Hill, New York.
			
			\bibitem{gr}
			I.~S.~Gradshteyn and I.~M.~Ryzhik, \emph{Table of Integrals, Series and Products}, 5th ed., Academic Press, San Diego, 1994.
			
			\bibitem{HL1}
			Hardy, G.~H.~and Littlewood, J.~E.~(1921). The zeros of Riemann's zeta-function on the critical line, \emph{Math.~Zeit}.~\textbf{10}, 283–317.
			
			\bibitem{HL2}
			Hardy, G.~H.~and Littlewood, J.~E. (1922). The approximate functional equation in the theory of the zeta-function, with applications to the divisor problems of Dirichlet and Piltz, \emph{Proc.~London Math.~Soc}.~(2) \textbf{21}, 39–74.
			
			\bibitem{HL3}
			Hardy, G.~H.~and Littlewood, J.~E.~(1929). The approximate functional equations for $\zeta(s)$ and $\zeta^2(s)$, \emph{Proc.~London Math.~Soc}.~(2) \textbf{29}, 81–97.
			
			
			\bibitem{riesz}
			Hardy, G. H.~and Riesz, M. (1915). The general theory of Dirichlet's series. \textit{Cambridge Tracts in Mathematics and Mathematical Physics}; no. 18.
			
			
			\bibitem{nakajima}
			Nakajima, M. (2003). A new expression for the product of the two Dirichlet series I. \textit{Proceedings of the Japan Academy}. Series A Mathematical Sciences, 79(2), 19–22.
			
			
			\bibitem{hurwitz}
			Wang, N. L. and Banerjee, S. (2017). On the product of Hurwitz zeta-functions. \emph{Proc.~Japan Acad, Ser.~A}, \textbf{93}, 31--36.
			
			\bibitem{wilton-2}
			Wilton, J.~R.~(1930). An approximate functional equation for the product of two $\zeta$-functions.~\emph{Proc.~London Math.~Soc.}~\textbf{31}, 11–17.
			
			
			
			
		\end{thebibliography}

	\end{document}